\documentclass[a4paper,12pt]{article}
\usepackage{amssymb,latexsym,amsmath,enumerate,verbatim,amsfonts,amsthm}
\usepackage[active]{srcltx}

\newtheorem{lemma}{Lemma}[section]
\newtheorem{definition}{Definition}[section]

\newtheorem{theorem}{Theorem}[section]
\newtheorem{corollary}{Corollary}[section]

\newtheorem{remark}{Remark}[section]
\newtheorem{example}{Example}[section]

\begin{document}
\title{{ \textsf{Robust SOS-Convex Polynomial Programs: Exact SDP Relaxations}\footnote{The authors are grateful to the referees and the handling editor for their constructive
comments and helpful suggestions which have contributed to the final preparation of
the paper. Research was partially supported by a grant from the Australian Research Council.}}}

\author{\textsc{V. Jeyakumar}\thanks{Corresponding author. Department of Applied Mathematics, University of New South Wales, Sydney 2052, Australia. E-mail: v.jeyakumar@unsw.edu.au}, \ \ \textsc{G. Li}\thanks{Department of Applied Mathematics, University of New South Wales, Sydney 2052, Australia. E-mail: g.li@unsw.edu.au} \ \ and \ \ \textsc{J. Vicente-P\'erez}\thanks{Department of Applied Mathematics, University of New South Wales, Sydney 2052, Australia. This author has been partially supported by the MICINN of Spain, Grant MTM2011-29064-C03-02. E-mail: jose.vicente@ua.es} }
\date{{\bf Revised Version}: September 28, 2013}

\maketitle

\vspace{-0.3cm}
\begin{abstract}
\noindent This paper studies robust solutions and semidefinite linear
programming (SDP) relaxations of a class of convex polynomial optimization problems in the face of data uncertainty.
 The class of convex optimization problems, called robust SOS-convex polynomial optimization problems, includes robust quadratically constrained convex optimization problems and robust separable convex polynomial optimization problems.
  It establishes sums-of-squares polynomial representations characterizing robust solutions and exact SDP-relaxations of robust SOS-convex polynomial optimization problems under various commonly used uncertainty sets. In particular, the results show that the polytopic and ellipsoidal uncertainty sets, that allow second-order cone re-formulations of robust quadratically constrained optimization problems, continue to permit exact SDP-relaxations for a broad class of robust SOS-convex polynomial optimization problems. 

\smallskip


\end{abstract}


\section{Introduction}

\label{SEC1}

%
%

Consider the convex polynomial optimization problem
\begin{equation*}
 \begin{array}{crl}
(P_0) & \inf & f(x) \\
     & s.t. & g_i(x)\leq 0,\ i=1,\ldots,m,
\end{array}
\end{equation*}
where $f$ and $g_i$'s are convex polynomials.
The model problem $(P_0)$ admits a hierarchy of semidefinite programming (SDP) relaxations, known as the Lasserre hierarchy of SDP-relaxations. More generally,
the Lasserre hierarchy of SDP relaxations is often used to solve  nonconvex polynomial optimization problems with compact feasible sets \cite{Lasserre,Lasserre2}, and it has finite convergence generically as shown recently in \cite{nie2012}.

In particular, if $f$ and $g_i, i=1,2,\ldots, m$ are SOS-convex polynomials (see Definition \ref{def:SOS_convex}), then $(P_0)$ enjoys an exact SDP-relaxation in the sense that the optimal values of $(P_0)$ and its relaxation problem are equal and the relaxation problem attains its optimum under the Slater constraint qualification (\cite[Theorem 3.4]{Lasserre2}). The class of SOS-convex polynomials is a numerically tractable subclass of convex polynomials and  it contains convex quadratic functions and convex separable polynomials \cite{Parrilo,HeNie10}. The SOS-convexity of a polynomial can be numerically checked by solving a semidefinite programming problem.

The exact SDP-relaxation of a convex optimization problem is a highly desirable feature because SDP problems can  be efficiently solved \cite{BV,las-hand}.  However, the data of real-world convex optimization problems are often uncertain (that is, they are not known exactly at the time of the decision) due to estimation errors, prediction errors or lack of information. Recently, robust optimization approach has emerged as a powerful approach to treat optimization under data uncertainty. It is known that a robust convex quadratic optimization problem under ellipsoidal data uncertainty enjoys exact SDP relaxation as it can be equivalently reformulated as a semi-definite programming problem (see \cite{robust}). In the same vein, Goldfarb and Iyengar \cite{goldfab} have
shown that a robust convex quadratic optimization problems under restricted ellipsoidal data uncertainty can be equivalently reformulated as a second-order cone programming problem.

Unfortunately, an exact SDP relaxation may fail for a robust convex (not necessarily quadratic) polynomial optimization problem under restricted ellipsoidal data uncertainty (see Example \ref{ex:3.1}). This raises the fundamental question: Do some classes of robust convex (not necessarily quadratic) polynomial optimization problems posses exact SDP relaxation? This question has motivated us to study SOS-convex polynomial optimization problems under uncertainty.

In this paper, we study the SOS-convex polynomial optimization problem $(P_0)$ in the face of data uncertainty. This model problem under data uncertainty in the constraints can be captured by the model problem
\begin{equation*}
\begin{array}{crl}
(UP_0) & \inf & f(x) \\
       & s.t. & g_i(x,v_i)\leq 0,\ i=1,\ldots,m,
\end{array}
\label{primal}
\end{equation*}
where $v_i$ is an uncertain parameter and $v_i \in \mathcal{V}_i$ for some compact uncertainty set $\mathcal{V}_i \in \mathbb{R}^{q_i}$, $q_i \in \mathbb{N}$, $f:\mathbb{R}^n\rightarrow\mathbb{R}$ is a SOS-convex polynomial and $g_i:\mathbb{R}^n\times\mathbb{R}^{q_i}\rightarrow\mathbb{R}$, $i=1,\ldots,m$, are functions such that for each $v_i\in\mathbb{R}^{q_i}$, $g_i(\cdot,v_i)$ is a SOS-convex polynomial. As solutions to convex optimization problems are generally sensitive to data uncertainty, even a small uncertainty in the data can affect the quality of the optimal solution of a convex optimization problem, making it far from an optimal solution and unusable from a practical viewpoint. Consequently, how to find robust optimal solutions, that are immunized against data uncertainty, has become an important question in convex optimization and has recently been extensively studied in the literature (see \cite{robust,BN1,jl-siam,jeya-boris,jeya-li-wang}).

Following the robust optimization approach, the robust counterpart of $(UP_0)$, which finds a robust solution of $(UP_0)$ that is immunized against all the possible uncertain scenarios, is given by \begin{equation*}
\begin{array}{crl}
(P) & \inf & f(x) \\
     & s.t. & g_i(x,v_i)\leq 0,\ \forall v_i\in \mathcal{V}_i,i=1,\ldots,m,
\end{array}
\label{robust}
\end{equation*}
and is called a \textit{robust SOS-convex polynomial optimization problem} or called simply, \textit{robust SOSCP}. In the robust counterpart, the uncertain inequality constraints are
enforced for all realizations of the uncertainties $v_i\in \mathcal{V}_i, \ i=1,\ldots,m$. A sum of squares (SOS) relaxation problem of $(P)$ with degree $k$ is the model problem
\begin{equation*}
\begin{array}{cr}
(D_k) & \sup\limits_{\mu\in \mathbb{R}, v_i\in \mathcal{V}_i, \lambda_i\geq 0} \left\{\mu\ :\ f(\cdot) +\sum\limits_{i=1}^{m}\lambda_i g_i(\cdot,v_i)-\mu \in \Sigma^2_k \right\} \\
\end{array}
\label{relaxed}
\end{equation*}
where $\Sigma^2_k$ denotes the set of all sum of squares polynomials with degree no larger than $k$. The model $(D_k)$ is, in fact, the sum of squares relaxation of the robust Lagrangian dual, examined recently in \cite{beck,jeya-goberna,jl-siam,jeya-boris,jeya-li-wang}.

The following contributions are made in this paper to robust convex optimization.

\smallskip

I. We first derive a complete characterization of the solution of a robust SOS-convex polynomial optimization problem $(P)$ in terms of sums of squares polynomials under a normal cone constraint qualification that is shown to be the weakest condition for the characterization. We show that the sum of squares characterization can be numerically checked for some classes of uncertainty sets by solving a semidefinite programming problem.

\smallskip

II. We establish that the value of a robust SOS-convex optimization problem (P) can be found by solving a sum-of-squares programming problem. This is done by proving an exact sum-of-squares relaxation of the robust SOS-convex optimization problem $(P)$.

\smallskip

III. Although the sum of squares relaxation problem $(D_k)$ is NP-hard for general classes of uncertainty sets, we prove that, for the classes of polytopic and ellipsoidal uncertainty sets, the relaxation problem can equivalently be re-formulated as a semidefinite programming problem. This shows that these uncertainty sets, which allow second-order cone re-formulations of robust quadratically constrained optimization problems \cite{goldfab}, permit exact SDP relaxations for a broad class of robust SOS-convex optimization problems. The relaxation problem provides an alternative formulation of an exact second-order cone relaxation for the robust quadratically constrained optimization problem, studied in \cite{goldfab}.

\medskip

\noindent The outline of the paper is as follows. Section 2 presents necessary and sufficient conditions for robust optimality and derives SOS-relaxation for robust SOS-convex optimizatin problems. Section 3 provides numerically tractable classes of robust convex optimization problems by presenting exact SDP-relaxations. 
\section{Solutions of Robust SOSCPs}

\label{SEC2}

We begin with some definitions and preliminaries on polynomials. We say that a real polynomial $f$ is sum-of-squares \cite{Laurent_survey} if there exist real polynomials $f_j$, $j=1,\ldots,r$, such that $f=\sum_{j=1}^rf_j^2$. The set consisting of all sum of squares real polynomial is denoted by $\Sigma^2$. Moreover, the set consisting of all sum of squares real polynomial with degree at most $d$ is denoted by $\Sigma^2_d$. The space of all real polynomials on $\mathbb{R}^n$
is denoted by $\mathbb{R}[x]$ and the set of all $n\times r$ matrix polynomials is denoted by $\mathbb{R}[x]^{n \times r}$.

\begin{definition}
{\bf (SOS matrix polynomial)} We say a matrix polynomial $F \in \mathbb{R}[x]^{n \times n}$ is a SOS matrix polynomial if $F(x)=H(x)H(x)^T$ where $H(x) \in \mathbb{R}[x]^{n \times r}$ is a matrix polynomial for some $r \in \mathbb{N}$.
\end{definition}
\begin{definition} \label{def:SOS_convex} {\bf (SOS-Convex polynomial} {\cite{HeNie10})}
 A real polynomial $f$ on $\mathbb{R}^n$ is called \textit{SOS-convex} if the Hessian matrix function $F:x\mapsto \nabla^2 f(x)$ is a SOS matrix polynomial.
\end{definition}
Clearly, a SOS-convex polynomial is convex. However, the converse is not true, that is, there exists a convex polynomial which is not SOS-convex \cite{Parrilo}.   It is known that any convex quadratic function and any convex separable polynomial is a SOS-convex polynomial. Moreover, a SOS-convex polynomial can be non-quadratic and non-separable. For instance, $f(x)=x_1^8+x_1^2+x_1x_2+x_2^2$ is a SOS-convex polynomial (see \cite{HeNie10}) which is  non-quadratic and non-separable.
\begin{lemma}[{\bf SOS-Convexity \& sum-of-squares} {\cite[Lemma 8]{HeNie10}}] \label{polysos}
Let $f$ be a SOS-convex polynomial on $\mathbb{R}^n$. If $f(u)=0$ and $\nabla f(u)=0$ for some $u\in\mathbb{R}^n$, then $f$ is a sum-of-squares polynomial.
\end{lemma}
The following existence result for  solutions of a convex polynomial optimization problem will also be useful for our later analysis.
\begin{lemma}[{\bf Solutions of convex polynomial optimization} {\cite[Theorem 3]{BeKl02}}]
\label{minattain}
Let $f_0,f_1,\ldots,f_m$ be convex polynomials on $\mathbb{R}^n$ and let $C:=\left\{x \in \mathbb{R}^n : f_i(x) \leq 0, i=1,\ldots,m\right\}$. If $\inf\limits_{x\in C}f_0(x)>-\infty$  then  $\operatorname{argmin}\limits_{x\in C}f_0(x) \neq \emptyset$.
\end{lemma}
We note that it is possible to reduce a convex polynomial optimization problem to a quadratic optimization problem by introducing new variables. For example, $\min_{x \in \mathbb{R}}\{x^2:x^4 \le 1\}$ can be
converted to a quadratic optimization problem $\min_{(x,t) \in \mathbb{R}^2}\{x^2:x^2 \le 1, t=x^2\}$. However, introducing new variables will result in a problem which may not satisfy the required convexity.

Recall that for a convex set $A\subset\mathbb{R}^{m}$, the normal cone of  $A$ at $x\in A$ is given by $N_A(x):=\left\{v\in \mathbb{R}^n : v^T(y-x)\leq 0,\,\forall y\in A\right\}.$
Let $F:=\left\{x : g_i(x,v_i)\leq 0\ \forall v_i\in \mathcal{V}_i, i=1,\ldots,m\right\} \neq \emptyset$. We say that the \emph{normal cone condition} holds for $F$ at $x\in F$ provided that $$ N_F(x) =  \left\{\sum_{i=1}^{m} \lambda_i \nabla_x g_i(x,v_i) : \lambda_i\geq 0, v_i\in \mathcal{V}_i, \lambda_i g_i(x,v_i)=0\right\},$$
where $\nabla_x$ denotes the gradient with respect to the variable $x$.

It is known from \cite{jl-siam} that the normal cone condition is guaranteed by the following robust Slater condition, $\left\{x\in\mathbb{R}^n : g_i(x,v_i)< 0, \ \forall v_i\in \mathcal{V}_i, i=1,\ldots,m\right\} \neq \emptyset.$
On the other hand, the normal cone condition is, in general, weaker than the robust Slater condition.

In the following theorem we first prove that the normal cone condition guarantees a robust solution characterization involving sums-of-squares representations for robust SOSCPs.

\begin{theorem}[{\bf Sum-of-squares characterization of solutions}]
\label{alternative2}
Let $f:\mathbb{R}^n\rightarrow\mathbb{R}$ be a SOS-convex polynomial and let $g_i:\mathbb{R}^n\times\mathbb{R}^{q_i}\rightarrow\mathbb{R}$, $i=1,\ldots,m$, be functions such that for each $v_i\in\mathbb{R}^{q_i}$, $g_i(\cdot,v_i)$ is a SOS-convex polynomial with degree at most $d_i$. Let $\mathcal{V}_i\subset \mathbb{R}^{q_i}$ be compact and $F:=\left\{x\in\mathbb{R}^n : g_i(x,v_i)\leq 0\ \forall v_i\in \mathcal{V}_i, i=1,\ldots,m\right\} \neq \emptyset$. Suppose that ${\rm argmin}_{x\in F}f(x)\neq\emptyset$ and the normal cone condition holds at $x^*\in F$. Then,
$x^*$ is a minimizer of $\min_{x \in F}f(x)$ if and only if $(\exists\ \bar{v}_i\in \mathcal{V}_i, \, \lambda_i \in \mathbb{R}_+, i=1,\ldots,m, \, \sigma_0\in \Sigma^2_{k_0})$ $(\forall \, x \in \mathbb{R}^n)$ \begin{equation}\label{eq:py}f(x)-f(x^*)+\sum_{i=1}^{m}\lambda_i g_i\left(x,\bar{v}_i\right) = \sigma_0(x),\end{equation} where $k_0$ is the smallest even number such that $k_0 \geq \max\left\{\deg f, \max_{1 \leq i \leq m}d_i\right\}$.
\end{theorem}

\begin{proof}

[(if part)]  It easily follows from the fact that $f(x)-f(x^*)=\sigma_0(x)-\sum_{i=1}^{m}\lambda_i g_i\left(x,\bar{v}_i\right) \geq 0$ for all $x\in F$ as $\sigma_0(x) \ge 0$ and $\lambda_i \ge 0$.
\smallskip
[(only if part)] If $f(x^*) = \min_{x\in F}f(x)$, then, by the optimality condition of convex optimization, $-\nabla f(x^*) \in N_F(x^*)$. By the normal cone condition, there exist $\bar{v}_i\in \mathcal{V}_i$, $\lambda_i\geq 0$, $i=1,\ldots,m$, with $\lambda_i g_i(x^*,\bar{v}_i) = 0$, such that $-\nabla f(x^*) = \sum_{i=1}^{m} \lambda_i \nabla g_i(x^*,\bar{v}_i)$. Let $L(x):=f(x)-f(x^*)+\sum_{i=1}^{m}{\lambda_i g_i(x,\bar{v}_i)}$, for $x\in\mathbb{R}^n$. Observe that $L(x^*)=0$ and $\nabla L(x^*)=0$. Clearly, $L$ is SOS-convex since $f$ and $g(\cdot,\bar{v}_i)$ are all SOS-convex. So, Lemma \ref{polysos} guarantees that $L$ is a sum of squares polynomial. Moreover, the degree of $L$ is not larger than $k_0$. So, there exists $\sigma_0 \in\Sigma^2_{k_0}$ such that
$f(x)-f(x^*)+\sum_{i=1}^{m}{\lambda_i g_i\left(x,\bar{v}_i\right)} = \sigma_0(x)  \quad \forall x\in\mathbb{R}^n.$ Thus, the conclusion follows.
\end{proof}

Next, we show that the normal cone condition is indeed a characterization for the robust solution characterization, in the sense that,  if the normal cone condition fails at some feasible point then there exists a SOS-convex real polynomial $f$ such that the robust solution characterization fails.
\begin{theorem}[{\bf Weakest qualification for solution characterization}]
\label{charact2}
Let $g_i:\mathbb{R}^n\times\mathbb{R}^{q_i}\rightarrow\mathbb{R}$, $i=1,\ldots,m$, be functions such that for each $v_i\in\mathbb{R}^{q_i}$, $g_i(\cdot,v_i)$ is a SOS-convex polynomial with degree at most $d_i$. Let $F:=\left\{x\in\mathbb{R}^n : g_i(x,v_i)\leq 0\ \forall v_i\in \mathcal{V}_i, i=1,\ldots,m\right\} \neq \emptyset$ and let $\mathcal{V}_i\subset \mathbb{R}^{q_i}$ be compact. Then, the following statements are equivalent:
\begin{enumerate}
	\item[{\rm (i)}] For each SOS-convex real polynomial $f$ on $\mathbb{R}^n$ with ${\rm argmin}_{x\in F}f(x)\neq\emptyset$,
{$f(x^*) = \min\limits_{x\in F}f(x)\ \Leftrightarrow \ \left[\exists\,\bar{v}_i\in \mathcal{V}_i, \lambda_i\geq 0  :  f(\cdot)-f(x^*)+\sum_{i=1}^{m}\lambda_i g_i\left(\cdot,\bar{v}_i\right) \in \Sigma^2_{k_0} \right]$}
where $k_0$ is the smallest even number such that $k_0 \geq \max\left\{\deg f, \max_{1 \leq i \leq m}d_i\right\}$.
	\item[{\rm (ii)}] $N_F(x) =  \left\{\sum_{i=1}^{m} \lambda_i \nabla_x g_i(x,v_i) : \lambda_i\geq 0, v_i\in \mathcal{V}_i, \lambda_i g_i(x,v_i)=0\right\}$,  for all $x\in F$.
\end{enumerate}
\end{theorem}

\begin{proof}
It suffices to show that (\emph{i}) $\Rightarrow$ (\emph{ii}) since the converse statement has been already shown in Theorem \ref{alternative2}. In fact, we just need to show
$ N_F(x) \subset \{\sum_{i=1}^{m} \lambda_i \nabla_x g_i(x,v_i) : \lambda_i\geq 0, v_i\in \mathcal{V}_i, \lambda_i g_i(x,v_i)=0\},$
for any $x\in F$, since the converse inclusion always holds. Let $x^*\in F$ be arbitrary. If
$ w\in N_F(x^*)$ then $-w^T(x-x^*)\geq 0$ for all $x\in F$. Let $f(x):=-w^T(x-x^*)$. Then, $\min_{x\in F}f(x) = f(x^*) = 0$. Since any affine function is SOS-convex, applying (\emph{i}), there exist $\bar{v}_i\in \mathcal{V}_i$, $\lambda_i\geq 0$, for $i=1,\ldots,m$, and $\sigma_0\in \Sigma^2_{k_0}$ such that, for all $x\in \mathbb{R}^n$,
\begin{equation}
 -w^T(x-x^*) + \sum_{i=1}^{m}\lambda_i g_i\left(x,\bar{v}_i\right) = \sigma_0(x) \geq 0.
\label{equ2}
\end{equation}
Letting $x=x^*$, we see that $\sum_{i=1}^{m}\lambda_i g_i\left(x^*,\bar{v}_i\right) \geq 0$. This together with $x\in F$ implies $\lambda_i g_i\left(x^*,\bar{v}_i\right) = 0$, $i=1,\ldots,m$. So, \eqref{equ2} implies that $\sigma_0(x^*)=0$ and $0=\nabla\sigma_0(x^*) = -w+\sum_{i=1}^{m}\lambda_i g_i\left(x^*,\bar{v}_i\right)$. Then,
$w\in \{ \sum_{i=1}^{m} \lambda_i \nabla_x g_i(x^*,v_i) : \lambda_i\geq 0, v_i\in \mathcal{V}_i, \lambda_i g_i(x^*,v_i)=0 \}.$
Thus, the conclusion follows.
\end{proof}

It is worth noting that the sum-of-squares condition characterizing the solution of a robust SOSCP can be numerically verified by solving semi-definite programming problems for some uncertainty sets. We illustrate this with two simple examples: (1) SOS-convex constraints with finite uncertainty sets; (2) quadratic constraints with spectral norm data uncertainty sets. The numerical tractability of more general classes of robust SOS-convex optimization problems under sophisticated classes of uncertainty sets  will be discussed later on in Section 3.

\subsection*{SOS-Convex constraints and finite uncertainty sets}
Suppose that $\mathcal{V}_i=\left\{v_i^1,\ldots,v_i^{s_i}\right\}$ for any $i\in\left\{1,\ldots,m\right\}$. Then, the robust SOS-convex polynomial optimization problem takes the form
\begin{equation*}
(P_1) \ \ \ \min_{x \in \mathbb{R}^n} \{f(x): g_i(x,v_i^j)\leq 0,\ \forall j=1,\ldots,s_i,\ \forall i=1,\ldots,m\},
\end{equation*}
and the minimum is attained in virtue of Lemma \ref{minattain}. Let $x^*$ be a feasible solution of $(P_1)$ and suppose that the normal cone condition is satisfied at $x^*$.

Let $k_0=\max_{1 \le i \le m, 1 \le j \le s_i}\{{\rm deg}f,{\rm deg}g_i(\cdot,v_i^j)\}$. In this case, (\ref{eq:py}) in Theorem \ref{alternative2} 
is equivalent to the condition that there exist ${\lambda}_i^j \geq 0$, $i=1,\ldots,m, j=1,\ldots,s_i$, and $\sigma_0\in \Sigma_k^2$ such that
\begin{equation}
\label{eq:0}
f(x)-f(x^*)+\sum_{1 \le i \le m, 1 \le j \le s_i}{\lambda}_i^j g_i\left(x, {v}_i^j\right) = \sigma_0(x) \mbox{\ for all\ } x\in\mathbb{R}^n.
\end{equation}
Indeed, it is easy to see that (\ref{eq:py}) in Theorem \ref{alternative2} implies \eqref{eq:0}. On the other hand, \eqref{eq:0} immediately gives us that $x^*$ is a solution of $(P_1)$ which implies (\ref{eq:py}). Therefore, a solution
of a SOS-convex polynomial optimization problem under finite uncertainty sets can be efficiently verified by solving a semidefinite programming problem. 

\subsection*{Quadratic constraints under spectral norm uncertainty}
Consider the following SOS-convex polynomial optimization problem with quadratic constraints under spectral norm uncertainty:
\begin{equation*}
  \min_{x \in \mathbb{R}^n} \{f(x) :x^TB_ix + 2b_i^T x + \beta_i \leq 0,\ i=1,\ldots,m\},
\end{equation*}
where, $b_i \in \mathbb{R}^n$ and $\beta_i \in \mathbb{R}$, the  data  $(B_i,b_i,\beta_i) \in S^n \times \mathbb{R}^n \times \mathbb{R}$, $i=1,\ldots,m$, are uncertain and belong to the spectral norm uncertainty set
 $$\mathcal{V}_i=\{(B_i,b_i,\beta_i) \in S^n \times \mathbb{R}^n \times \mathbb{R} : \|\left(\begin{array}{cc}
B_i & b_i\\
b_i^T & \beta_i
\end{array} \right)-\left(\begin{array}{cc}
\overline{B}_i & \overline{b}_i\\
\overline{b}_i^T & \overline{\beta}_i
\end{array} \right)\|_{\rm spec} \leq \varepsilon_i\},$$ for some $\varepsilon_i \ge 0$, $\overline{B}_i \succeq 0$, $\overline{b}_i \in \mathbb{R}^n$ and $\overline{\beta}_i \in \mathbb{R}$.
Here, $S^{n}$ denotes the space of symmetric $n\times n$ matrices and $\|\cdot\|_{\rm spec}$ denotes the spectral norm defined by
$\|M\|_{\rm spec}=\sqrt{\lambda_{\max}(M^TM)}$ where $\lambda_{\max}(C)$ is the maximum eigenvalue of the matrix $C$. The corresponding robust counterpart of the above problem is
\begin{equation*}
\begin{array}{crl}
(P_2) & \min & f(x) \\
    & s.t. & x^TB_ix + 2b_i^T x + \beta_i \leq 0, \ \forall \, (B_i,b_i,\beta_i) \in \mathcal{V}_i, \ i=1,\ldots,m,
\end{array}
\end{equation*}
Let $k=\max\{{\rm deg} f,2\}$.
In this case, (\ref{eq:py}) in Theorem \ref{alternative2} becomes $(\exists\,B_i\in \mathcal{V}_i,\,\lambda \in \mathbb{R}^m_+,\,\sigma_0\in\Sigma_k^2)$ $(\forall x \in \mathbb{R}^n)$
 $f(x)-f(x^*)+\sum_{i=1}^{m}\lambda_i (x^TB_ix + 2b_i^T x + \beta_i ) = \sigma_0(x)$.
This, in turn, is equivalent to the condition that there exist $\lambda_i \geq 0$, $i=1,\ldots,m$, and $\sigma_0\in \Sigma_k^2$ such that
\begin{equation}
\label{eq:001}
f(x)-f(x^*)+\sum_{i=1}^{m}\lambda_i (x^T(\overline{B}_i+\varepsilon_i I_n) x + 2\overline{b}_i^T x + \overline{\beta}_i+\varepsilon_i ) = \sigma_0(x) \mbox{\ for all\ } x\in\mathbb{R}^n.
\end{equation}
In fact, \eqref{eq:001} implies (\ref{eq:py})  as $(\overline{B}_i+\varepsilon_i I_n,\overline{b}_i,\overline{\beta}_i+\varepsilon_i) \in \mathcal{V}_i$. On the other hand, note that, for all $(B_i,b_i,\beta_i) \in \mathcal{V}_i$,  $\left(\begin{array}{cc}
\overline{B}_i+\varepsilon_i I_n & \overline{b}_i\\
\overline{b}_i^T & \beta_i+\varepsilon_i
\end{array} \right)-\left(\begin{array}{cc}
B_i & b_i\\
b_i^T & \beta_i
\end{array} \right)$ is a positive semidefinite matrix, and hence, for each $i=1,\ldots,m$, $$h_i(x):=  \left(\begin{array}{c}
x\\
1
\end{array} \right)^T(\left(\begin{array}{cc}
\overline{B}_i+\varepsilon_i I_n & \overline{b}_i\\
\overline{b}_i^T & \beta_i+\varepsilon_i
\end{array} \right)-\left(\begin{array}{cc}
B_i & b_i\\
b_i^T & \beta_i
\end{array} \right))\left(\begin{array}{c}
x\\
1
\end{array} \right)$$ is sum-of-squares.
So, (\ref{eq:py}) implies that there exist $\lambda_i \ge 0$ and $(B_i,b_i,\beta_i) \in \mathcal{V}_i$,
such that \begin{eqnarray*}
&  & f(x)-f(x^*)+\sum_{i=1}^{m}\lambda_i (x^T(\overline{B}_i+\varepsilon_i I_n) x + 2\overline{b}_i^T x + \overline{\beta}_i+\varepsilon_i ) \\
&= & f(x)-f(x^*)+\sum_{i=1}^{m}\lambda_i (x^TB_ix + 2b_i^T x + \beta_i )+\sum_{i=1}^mh_i(x),  
\end{eqnarray*}
is a sum-of-squares polynomial with degree at most $k$. Therefore, a solution
of a quadratic optimization problem under spectral norm data uncertainty can also be efficiently verified by solving a semidefinite programming problem.

Next, we examine how to find
the optimal value of a robust SOSCP by solving a sum of squares relaxation problem. In particular,  the corresponding sum of squares relaxation
problem can often be equivalently reformulated as  semi-definite programming problems under various commonly used data uncertainty sets.

\begin{theorem}[{\bf Exact sum of squares relaxation}]
\label{alternative}
Let $f:\mathbb{R}^n\rightarrow\mathbb{R}$ be a SOS-convex polynomial.  Let $g_i:\mathbb{R}^n\times\mathbb{R}^{q_i}\rightarrow\mathbb{R}$, $i=1,\ldots,m$, be functions such that for each $x \in \mathbb{R}^n$, $g_i(x,\cdot)$ is
concave, $g_i(\cdot,v_i)$ is a SOS-convex polynomial for each $v_i\in \mathcal V_i$ with degree at most $d_i$, and $\mathcal V_i \subset \mathbb{R}^{q_i}$ are convex compact sets.  Let
$\left\{x\in\mathbb{R}^n : g_i(x,v_i) < 0\ \forall v_i\in \mathcal{V}_i, i=1,\ldots,m\right\} \neq \emptyset$. Then, we have
$$\inf\{f(x): g_i(x,v_i)\leq 0\ \forall v_i\in \mathcal{V}_i, i=1,\ldots,m\} = \max_{\lambda_i \geq 0, v_i \in \mathcal{V}_i}\{ \mu:  f(\cdot)+\sum_{i=1}^{m}{{\lambda}_i g_i\left(\cdot,{v}_i\right)}-\mu  \in \Sigma_{k_0}^2 \},$$
where $k_0$ is the smallest even number such that $k_0 \geq \max\left\{\deg f, \max_{1 \leq i \leq m}d_i\right\}$.
\end{theorem}

\begin{proof}
Note that, for any $\lambda_i \geq 0$, $v_i \in \mathcal{V}_i$ with $f(\cdot)+\sum_{i=1}^{m}{{\lambda}_i g_i\left(\cdot,{v}_i\right)}-\mu \in \Sigma_{k_0}^2$ and any point $x \in \mathbb{R}^n$ such that $g_i(x,v_i) \leq 0$ for all $v_i\in \mathcal{V}_i$, one has $f(x) \geq f(x)+\sum_{i=1}^{m}{{\lambda}_i g_i\left(x,{v}_i\right)} \geq \mu$. So, we see that $\inf\{f(x): g_i(x,v_i)\leq 0\ \forall v_i\in \mathcal{V}_i,\ i=1,\ldots,m\} \geq \max_{\lambda_i \geq 0, v_i \in \mathcal{V}_i}\{ \mu:  f(\cdot)+\sum_{i=1}^{m}{{\lambda}_i g_i\left(\cdot,{v}_i\right)}-\mu  \in \Sigma_{k_0}^2 \}.$
\smallskip

To see the reverse inequality, we may assume without loss of generality that $c:=\inf\{f(x): g_i(x,v_i)\leq 0\ \forall v_i\in \mathcal{V}_i,\ i=1,\ldots,m\} \in \mathbb{R}$. 
By the usual convex programming duality and the robust Slater condition,
\begin{eqnarray*}
\inf\{f(x):g_i(x,v_i) \le 0, \, \forall \, v_i \in \mathcal{V}_i\}   &=& \inf\{f(x): \max_{v_i \in \mathcal{V}_i} g_i(x,v_i) \le 0,  i=1,2,\ldots, m\} \\
 = \quad \max\limits_{\substack{\lambda_i\geq 0}}\:\inf_{x\in\mathbb{R}^n}\left\{f(x)+\sum_{i=1}^{m}\lambda_i \max_{v_i \in \mathcal{V}_i}g_i(x,v_i)\right\} & = &  \max\limits_{\substack{\lambda_i\geq 0}}\:\inf_{x\in\mathbb{R}^n} \max_{v_i \in \mathcal{V}_i}\left\{f(x)+\sum_{i=1}^{m}\lambda_i g_i(x,v_i)\right\},
\end{eqnarray*}
where the attainment of the maximum is guaranteed by the robust Slater condition.

Note that $x \mapsto g_i(x,v_i)$ is SOS-convex (and so convex) and $v_i \mapsto g_i(x,v_i)$ is concave. From the convex-concave minimax theorem, we have, for each $\lambda_i \ge 0$,
$$ \inf_{x\in\mathbb{R}^n} \max_{v_i \in \mathcal{V}_i}\left\{f(x)+\sum_{i=1}^{m}\lambda_i g_i(x,v_i)\right\}=\max\limits_{v_i\in \mathcal{V}_i} \inf_{x\in\mathbb{R}^n}\left\{f(x)+\sum_{i=1}^{m}\lambda_i g_i(x,v_i)\right\}. $$
So, $\inf\{f(x):g_i(x,v_i) \le 0, \, \forall \, v_i \in \mathcal{V}_i\} = \max\limits_{v_i\in \mathcal{V}_i, \lambda_i \ge 0} \inf_{x\in\mathbb{R}^n}\left\{f(x)+\sum_{i=1}^{m}\lambda_i g_i(x,v_i)\right\}   . $
Hence,  there exist $\bar{v}_i\in \mathcal{V}_i$ and $\bar{\lambda}_i\geq 0$, for $i=1,\ldots,m$, such that $f(x) + \sum_{i=1}^{m}\bar{\lambda}_i g_i\left(x,\bar{v}_i\right) \geq c$ for all $x\in\mathbb{R}^n.$

Let $h(x):=f(x) +  \sum_{i=1}^{m}{\bar{\lambda}_i g_i\left(x,\bar{v}_i\right)} - c $. Then, $h\geq 0$ and it is also a SOS-convex polynomial, as $f(\cdot)$ and $g\left(\cdot,\bar{v}_i\right)$ are all SOS-convex. So, by Lemma \ref{minattain}, we obtain that $\min_{x\in\mathbb{R}^n}h(x) = h(x^*)$ for some $x^*\in \mathbb{R}^n$. The polynomial $L(x):=h(x)-h(x^*)$ is again SOS-convex. Moreover, $L(x^*)=0$ and $\nabla L(x^*)=0 $. Then, $L$ is a sum of squares polynomial as a consequence of Lemma \ref{polysos}. So we get that
$f(x)+\sum_{i=1}^{m}{\bar{\lambda}_i g_i\left(x,\bar{v}_i\right)} - c - h(x^*) = \sigma_1(x)  \quad  \forall x\in\mathbb{R}^n,$
where $\sigma_1 \in\Sigma^2_{k_0}$ and $k_0$ is the smallest even number such that $k_0 \geq \max\left\{\deg f, \max_{1 \leq i \leq m}d_i\right\}$. Note that a sums-of-squares polynomial must be
of even degree. As $h(x^*) \geq 0$, $\sigma_0(\cdot) := \sigma_1(\cdot) + h(x^*)$ is also a sum of squares with degree at most $k_0$. Therefore, $\sigma_0 \in \Sigma^2_{k_0} $ and
$f(x)+\sum_{i=1}^{m}{\bar{\lambda}_i g_i\left(x,\bar{v}_i\right)} - c  = \sigma_0(x)$ for all  $x\in\mathbb{R}^n.$
Hence, $c \leq \displaystyle \max_{\lambda_i \geq 0, v_i \in \mathcal{V}_i}\{ \mu:  f(\cdot)+\sum_{i=1}^{m}{{\lambda}_i g_i\left(\cdot,{v}_i\right)}-\mu \in \Sigma_{k_0}^2 \}$ and the conclusion follows.
\end{proof}

\begin{remark}[{\bf Intractability of general SOS-relaxation problems}]
In general, finding the optimal value of a robust SOSCP using the SOS-relaxation problem can still be an intrinsically hard problem. For example, consider the following robust convex quadratic
optimization problem:
\begin{eqnarray*}
(P_4) & \inf  & x^TAx + 2a^T x +\alpha   \\
      &  s.t. & g_i(x,v_i) \leq 0, \quad \forall\, v_i^T Q_i^l v_i \leq 1, l=1,\ldots,k, i=1,\ldots,m,
\end{eqnarray*}
where
\begin{equation}
\label{eq:04}
g_i(x,v_i)=\|(\overline{B}_i^0+\sum_{j=1}^s v_i^j \overline{B}_i^j )x\|^2 + 2(\overline{b}_i^0+\sum_{j=1}^s v_i^j \overline{b}_i^j)^Tx + (\overline{\beta}_i^0+\sum_{j=1}^s v_i^j \overline{\beta}_i^j ),
\end{equation}
$v_i=(v_i^1,\ldots,v_i^s)$, $Q_i^l \succ 0$, for $l=1,\ldots,k$, $i=1,\ldots,m$, and $k \geq 2$. It was shown in \cite[Section 3.2.2]{BN1} that checking the robust feasibility of the problem $(P_4)$, is an NP-hard problem even under the robust Slater condition. 
By applying Theorem \ref{alternative} with $f(x)=0$, we see that the robust feasibility problem of $(P_4)$ is equivalent to the condition that the optimal value of the following problem is zero:
\[  \sup\limits_{\mu\in \mathbb{R}, v_i^TQ_i^lv_i \le 1, \lambda_i\geq 0} \left\{\mu\ :\ \sum_{i=1}^{m}\lambda_i g_i(\cdot,v_i)-\mu \in \Sigma^2_2 \right\}.  \]
This, in particular, shows that finding the optimal value of a SOS-relaxation of a robust SOSCP, via Theorem \ref{alternative}, is also NP-hard, and so, cannot be equivalently rewritten as a semi-definite programming problem.

Furthermore, observe that for the above intractable case, $v_i \mapsto g_i(x,v_i)$ with $g_i(x,v_i)$ defined in \eqref{eq:04}
is not affine.
\end{remark}

A semidefinite programming approximation
scheme for solving a robust nonconvex polynomial optimization problem has been given in \cite{Lasserre_robust} where it was shown that the optimal value
of the robust optimization problem can be approached as close as possible by the optimal value of a sequence of semi-definite programming relaxation problem under mild conditions. In the next section, we show that, in the case of affine data parametrization  (that is, $v_i \mapsto g_i(x,v_i)$ is affine), the optimal value of the robust SOS-convex polynomial optimization problem can be found by solving {\it a single semi-definite programming problem} under two commonly used data uncertainty sets: polytopic uncertainty and ellipsoidal uncertainty.

\setcounter{equation}{0}
\section{Exact SDP-Relaxations \& Affine Parameterizations}
In this section we consider the robust SOSCP under affinely parameterized data uncertainty:
\begin{equation*}
\begin{array}{crl}
(P) & \inf & f(x) \\
     & \mbox{s.t.} & g_i(x,v_i)\leq 0,\ \forall v_i\in \mathcal{V}_i,i=1,\ldots,m,
\end{array}
\end{equation*}
where $f$ is a SOS-convex polynomial and the data is affinely parameterized in the sense that \begin{equation*}
g_i(\cdot,v_i)=g_i^{(0)}(\cdot) + \sum_{r=1}^{q_i} v_i^{(r)}g_i^{(r)}(\cdot),  \quad v_i=(v_i^{(1)},\ldots,v_i^{(t_i)},v_i^{(t_i+1)},\ldots,v_i^{(q_i)}) \in \mathcal{V}_i \subseteq \mathbb{R}_{+}^{t_i} \times \mathbb{R}^{q_i-t_i},
\end{equation*}
where $g_i^{(r)}$, for $r=1,\ldots,t_i$, are SOS-convex polynomials, and $g_i^{(r)}$, for $r=t_i+1,\ldots,q_i$, are affine functions, for all $i=1,\ldots,m$.

In the following, we show, in the case of two commonly used uncertain sets, that the SOS-relaxation is exact and the relaxation problems can be represented as semi-definite programming (SDP) problems
whenever a robust Slater condition is satisfied.

\subsection{Polytopic data uncertainty}
Consider the robust SOSCP with polytopic uncertainty, that is, $t_i=q_i$ and $\mathcal{V}_i=\bar{\mathcal{V}}_i$, where $\bar{\mathcal{V}}_i$ is given by \begin{equation*}
\bar{\mathcal{V}}_i :=\{v_i=(v_i^{(1)},\ldots,v_i^{(q_i)})\in \mathbb{R}^{q_i}: v_i^{(r)} \geq 0,  A_i v_i = b_i\},
\end{equation*}
for some matrix $A_i=(A_i^{jr}) \in \mathbb{R}^{l_i \times q_i}$ and $b_i=(b_i^j) \in \mathbb{R}^{l_i}$ such that $\bar{\mathcal{V}}_i$ is compact. Examples include the case where $\mathcal{V}_i$ is a simplex, i.e., $ \mathcal{V}_i = \{v_i \in \mathbb{R}^{q_i} : v_i^{(r)} \geq 0, \sum_{r=1}^{q_i}v_i^{(r)}=1\}$ or more generally where $A_i\in \mathbb{R}^{l_i \times q_i}$ and  $\{x \in \mathbb{R}^{q_i}:A_ix=0\} \cap \mathbb{R}_+^{q_i}=\{0\}$.

For the robust SOSCP $(P)$ with polytopic uncertainty sets $\bar{\mathcal{V}}_i$, named $(P^p)$, and each $k \in \mathbb{N}$, the corresponding relaxation problem $(D_k^p)$ can be stated as
\begin{equation}
\label{eq:06}
\begin{array}{ccl}
(D_k^p)  &  \max\limits_{\mu, w_i^r}  &  \mu  \\
         &   s.t.   &   f +  \sum\limits_{i=1}^{m} \sum\limits_{r=0}^{q_i} w_i^r g_i^{(r)} - \mu \in \Sigma^2_{k}  \\
				 &          &   \sum\limits_{r=1}^{q_i} A_{i}^{jr} w_i^r = w_i^0 b_i^j, \ \forall j=1,\ldots,l_i, i=1,\ldots,m, \\
				 &          &   \mu\in\mathbb{R}, w_i^r \geq 0, \ \forall r=0,1,\ldots,q_i, i=1,\ldots,m.
\end{array}
\end{equation}
Let $k_0$ be the smallest even number such that $k_0 \geq \max_{0 \leq r \leq q_i, 1 \leq i \leq m}\{\deg f,\deg g_i^{(r)}\}$. Note that the relaxation problem $(D_k^p)$ can be equivalently rewritten as a semidefinite programming problem.

\begin{theorem}[{\bf Exact SDP-relaxation under polytopic data uncertainty}]\label{cor:1}
Consider the uncertain SOS-convex polynomial optimization problem under polytope data uncertainty $(P^p)$ and its relaxation problem $(D_k^p)$.
Suppose that
$$\left\{ x\in\mathbb{R}^n : g_i^{(0)}(x)+\sum_{r=1}^{q_i}v_i^{(r)}g_i^{(r)}(x)<0\ \forall v_i\in \bar{\mathcal{V}}_i, i=1,\ldots,m\right\} \neq \emptyset.$$
Then, the minimum of $(P^p)$ is attained and $\min(P^p)=\max(D_{k_0}^p)$,
where $k_0$ is the smallest even number such that $k_0 \geq \max_{0 \leq r \leq q_i, 1 \leq i \leq m}\{\deg f,\deg g_i^{(r)}\}$.
\end{theorem}

\begin{proof}
Denote the extreme points of $\bar{\mathcal{V}}_i$ by $v_i^1,\ldots,v_i^{s_i}$, $i=1,\ldots,m$. As $v_i \mapsto g_i(x,v_i)$ is affine, $\max_{v_i \in \bar{\mathcal{V}}_i}g_i(x,v_i)=\max_{1 \leq j \leq s_i}g_i(x,v_i^j)$ for each fixed $x \in \mathbb{R}^n$. So, $\min(P^p)$ can be equivalently rewritten as follows:
$\min\limits_{x\in\mathbb{R}^n}\left\{ f(x) : g_i(x,v_{i}^j)\leq 0, \forall j=1,\ldots,s_i,\,\forall i=1,\ldots,m\right\}$.
So, the minimum in the primal problem is attained in virtue of Lemma \ref{minattain}.


%

Now, according to Theorem \ref{alternative}, we have
\begin{eqnarray}
\label{dualpolyA}
\min(P^p) & = & \min\limits_{x\in\mathbb{R}^n} \left\{ f(x) : g_i^{(0)}(x)+\sum_{r=1}^{q_i} v_i^{(r)}g_i^{(r)}(x) \leq 0 \ \forall v_i\in \bar{\mathcal{V}}_i, i=1,\ldots,m\right\} \nonumber \\
          & = & \max\limits_{\substack{\mu\in \mathbb{R}, \lambda_i \ge 0 \\ v_i^{(r)} \ge 0, A_i v_i = b_i}} \left\{\mu\ :\ f + \sum\limits_{i=1}^{m} \lambda_i \left(g_i^{(0)}+\sum_{r=1}^{q_i} v_i^{(r)}g_i^{(r)} \right) - \mu \in \Sigma^2_{k_0} \right\}.
\end{eqnarray}
Let $w_i^{0}:=\lambda_i$ and $w_i^r:=\lambda_i v_i^{(r)}$ for $r=1,\ldots,q_i$. Observe that, 
for each $i=1,\ldots,m$,
$$\lambda_i \geq 0, v_i^{(r)} \geq 0\ \forall r=1,\ldots,q_i, A_i v_i = b_i $$
is equivalent to
$ w_i^r \geq 0 \ \forall r=0,1,\ldots,q_i, \sum_{r=1}^{q_i} A_{i}^{jr}w_i^r = w_i^0 b_i^j \ \forall j=1,\ldots,l_i.$
So, the maximization problem in \eqref{dualpolyA} collapses to that one in \eqref{eq:06} with $k=k_0$. Thus,  $\min(P^p)=\max(D^p_{k_0})$.
\end{proof}

The above theorem illustrates that, under the robust strict feasibility condition, an exact SDP relaxation holds for robust SOS-convex polynomial
optimization problem under polytopic data uncertainty. In the special case of robust convex quadratic optimization problem, such an exact SDP relaxation result was given in \cite{BN1,BV}.

\subsection{Restricted ellipsoidal data uncertainty}

Consider the robust SOSCP with a restrictive ellipsoidal uncertainty, that is, $\mathcal{V}_i=\hat{\mathcal{V}}_i$ where $\hat{\mathcal{V}}_i$ is given by
\begin{equation}\label{unC}
\hat{\mathcal{V}}_i :=\{ v_i \in \mathbb{R}^{q_i} : v_i^{(r)} \geq 0, r=1,\ldots,t_i, \|(v_i^{(1)},\ldots,v_i^{(t_i)})\| \leq 1, \|(v_i^{(t_i+1)},\ldots,v_i^{(q_i)})\| \leq 1\}.
\end{equation}
In the case where all the SOS-convex polynomials are convex quadratic functions, the above problem collapses to the robust quadratic optimization problem under restrictive ellipsoidal uncertainty set which was examined in \cite{goldfab}.
It is worth noting that the restriction of $v_i^{(r)} \ge 0$ is essential. Indeed, as pointed out in \cite{goldfab}, if this nonnegative restriction
is dropped, the corresponding robust quadratic optimization problem becomes NP-hard. It should also be noted that SOS-convexity of $g_i(\cdot, v)$ may not be preserved if $v_i^{(r)}$ is negative for some $r$ because $g_i(\cdot,v)=g_i^{(0)}+\sum_{r=1}^{q_i} v_i^{(r)}g_i^{(r)}$ and $g_i^{(r)}$'s are all SOS-convex polynomials.

We begin this case by providing a numerically tractable characterization for a point $x$ to be feasible for the robust SOSCP under the above restricted ellipsoidal data uncertainty.
\begin{lemma}[{\bf Robust feasibility characterization}]
Let $x \in \mathbb{R}^n$ and let $\hat{\mathcal{V}}_i$ be given in (\ref{unC}). Then $x$ is feasible for the robust SOSCP problem (P), that is,  $g_i^{(0)}(x)+\sum_{r=1}^{q_i} v_i^{(r)}g_i^{(r)}(x) \leq 0$,  $\forall v_i\in \hat{\mathcal{V}}_i$, $i=1,2, \ldots, m$, if and only if,
for each $i=1,\ldots,m$, the following second-order cone programming problem has a non-negative optimal value
\begin{equation*}
\begin{array}{cl}
 \max\limits_{\mu_1,\mu_2,\lambda_i^{(r)}}  &  -a_i^{(0)}-\mu_1-\mu_2 \\
  s.t.  &  \|\big(a_i^{(1)}+\lambda_i^{(1)},\ldots, a_i^{(t_i)}+\lambda_i^{(t_i)}\big)\|\le \mu_1, \\
        &  \|\big(a_i^{(t_i+1)},\ldots, a_i^{(q_i)}\big)\|\le \mu_2, \\
				&  \mu_1, \mu_2 \geq 0, \lambda_i^{(r)}\geq 0, \ \forall r=1,\ldots,t_i,
\end{array}
\end{equation*}
where $a_i^{(r)}=g_i^{(r)}(x)$, $r=0,1,\ldots,q_i$.
\end{lemma}

\begin{proof}
Fix $i \in \{1,\ldots,m\}$. Note that $x$ is feasible for robust SOSCP under the restricted ellipsoidal uncertainty is equivalent to
$$ \left. \begin{array}{c}
\|(v_i^{(1)},\ldots,v_i^{(t_i)})\| \leq 1, \, v_i^{(r)} \geq 0, \, r=1,\ldots,t_i \\
\|(v_i^{(t_i+1)},\ldots,v_i^{(q_i)})\| \leq 1
\end{array} \right\}
\Rightarrow -g_i^{(0)}(x)-\sum_{r=1}^{q_i} v_i^{(r)} g_i^{(r)}(x) \geq 0. $$
Using the standard Lagrangian duality theorem, this can be equivalently rewritten as \begin{equation*}
0  \leq  \inf\limits_{v_i \in \mathbb{R}^{q_i}} \left\{ -g_i^{(0)}(x)-\sum_{r=1}^{q_i} v_i^{(r)} g_i^{(r)}(x) :
\begin{array}{l}
\|(v_i^{(1)},\ldots,v_i^{(t_i)})\| \leq 1, \, v_i^{(r)} \geq 0, \, r=1,\ldots,t_i \\
\|(v_i^{(t_i+1)},\ldots,v_i^{(q_i)})\| \le 1
\end{array}
\right\}
\end{equation*}
\begin{equation*}
= \max\limits_{\substack{\mu_1,\mu_2 \ge 0 \\ \lambda_i \in\mathbb{R}^{t_i}_+ }}  \  \inf\limits_{v_i \in \mathbb{R}^{q_i}}   \   \left\{
\begin{array}{c}
-g_i^{(0)}(x)-\sum\limits_{r=1}^{q_i} v_i^{(r)} g_i^{(r)}(x) + \mu_1(\|(v_i^{(1)},\ldots,v_i^{(t_i)})\| -1) \\
 +\mu_2(\|(v_i^{(t_i+1)},\ldots,v_i^{(q_i)})\|-1)-\sum\limits_{r=1}^{t_i}\lambda_i^{(r)}v_i^{(r)} \
\end{array}
\right\}
\end{equation*}
\begin{equation*}
= \max\limits_{\substack{\mu_1,\mu_2 \ge 0 \\ \lambda_i \in\mathbb{R}^{t_i}_+ }}  \ \left\{ -g_i^{(0)}(x)-\mu_1-\mu_2  :
\begin{array}{l}
\|\big(g_i^{(1)}(x)+\lambda_i^{(1)},\ldots, g_i^{(t_i)}(x)+\lambda_i^{(t_i)}\big)\|\le \mu_1, \\
 \|\big(g_i^{(t_i+1)}(x),\ldots, g_i^{(q_i)}(x)\big)\|\le \mu_2.
\end{array}
\right\}
\end{equation*}
Hence, the equivalence follows.
\end{proof}

For the robust SOSCP $(P)$ with restricted ellipsoidal uncertainty sets $\hat{\mathcal{V}}_i$, named $(P^e)$, and each $k \in \mathbb{N}$, the corresponding relaxation problem $(D_k^e)$ can be stated as

\begin{equation}
\label{eq:08}
\begin{array}{ccl}
(D_k^e)  &  \max\limits_{\mu, w_i^r}  &  \mu  \\
         &   s.t.   &   f + \sum\limits_{i=1}^{m} \sum\limits_{r=0}^{q_i} w_i^r g_i^{(r)} - \mu \in \Sigma^2_{k}  \\
			 &          &   \|(w_i^{1},\ldots,w_i^{t_i})\| \leq w_i^{0}, \ \forall i=1,\ldots,m, \\
  			 &          &   \|(w_i^{t_i+1},\ldots,w_i^{q_i})\| \leq w_i^{0}, \ \forall i=1,\ldots,m, \\
				 &          &   w_i^r \geq 0, \ \forall r=0,1,\ldots,t_i, \ \forall i=1,\ldots,m. \\
				 &          &   \mu\in\mathbb{R}, w_i^r \in \mathbb{R}, \ \forall r=t_i+1\ldots,q_i,\ \forall i=1,\ldots,m.
\end{array}
\end{equation}
Let $k_0$ be the smallest even number such that $k_0\geq \max_{0 \leq r \leq q_i, 1 \leq i \leq m}\{ \deg f, \deg g_i^{(r)}\}$. As in the polytopic case, the relaxation problem $(D_k^e)$ can be equivalently rewritten as a semidefinite programming problem with an additional second-order cone constraint. 


\begin{theorem}[{\bf Exact SDP-relaxation under restricted ellipsoidal  uncertainty}]
\label{cor43}
Consider the uncertain SOS-convex polynomial optimization problem under ellipsoidal data uncertainty $(P^e)$ and its relaxation problem $(D_k^e)$.
Suppose that
 $\{x\in\mathbb{R}^n : g_i^{(0)}(x)+\sum_{r=1}^{q_i} v_i^{(r)} g_i^{(r)}(x)<0 \ \forall v_i\in \hat{\mathcal{V}}_i, i=1,\ldots,m\} \neq \emptyset.$
Then, $\inf(P^e)=\max(D_{k_0}^e)$. 
\end{theorem}

\begin{proof}
Using Theorem \ref{alternative}, we obtain that
\begin{eqnarray}
\label{dualpoly1}
\inf(P^e) & = & \inf\limits_{x\in\mathbb{R}^n}\left\{ f(x) : g_i^{(0)}(x)+\sum_{r=1}^{q_i} v_i^{(r)}g_i^{(r)}(x) \leq 0, \  \forall v_i\in \hat{\mathcal{V}}_i, i=1,\ldots,m\right\} \nonumber \\
          & = & \max\limits_{\substack{ \mu\in\mathbb{R}, \lambda_i\geq 0 \\ v_i\in \hat{\mathcal{V}}_i}} \left\{\mu : f + \sum\limits_{i=1}^{m}\lambda_i g_i^{(0)}  + \sum\limits_{i=1}^{m}\sum_{r=1}^{q_i} \lambda_i v_i^{(r)} g_i^{(r)}  -\mu \in \Sigma^2_{k_0} \right\}.
\end{eqnarray}
To see the conclusion, define $w_i^{0}:=\lambda_i$ and $w_i^{r}:=\lambda_i v_i^{(r)}$, $r=1,\ldots,q_i$. It can be verified that, for each $i=1,\ldots,m$,
$\lambda_i \geq 0,v_i\in \hat{\mathcal{V}}_i $
is equivalent to
\begin{eqnarray*}
\|(w_i^{1},\ldots,w_i^{t_i})\| & \leq & w_i^{0}, \quad w_i^r \geq 0, \ \forall r=0,1,\ldots,t_i, \\
\|(w_i^{t_i+1},\ldots,w_i^{q_i})\| & \leq & w_i^{0}, \quad w_i^r \in \mathbb{R}, \ \forall r=t_i+1\ldots,q_i.
\end{eqnarray*}
So, the maximization problem in \eqref{dualpoly1} collapses to that one in \eqref{eq:08} with $k=k_0$. Thus, we get $\inf(P^e)=\max(D^e_{k_0})$.
\end{proof}

In the following example, we show that an exact SDP relaxation may fail for a robust convex (but not SOS-convex) polynomial optimization problem with linear constraints under restricted ellipsoidal data uncertainty.

\begin{example}{\bf (Failure of exact SDP relaxation for convex polynomial optimization)}\label{ex:3.1}
Let $f$ be a convex homogeneous polynomial with degree at least $2$ in $\mathbb{R}^n$  which is not a sum-of-squares polynomial (see \cite{Laurent1,convexA}, for the existence  of such polynomials).  Consider the following robust convex polynomial optimization problem under ellipsoidal data uncertainty:
\begin{eqnarray*}
& \min_{x \in \mathbb{R}^n} & f(x) \\
& \mbox{ s.t. } & v^Tx -1 \le 0, \, \forall \, v\in \mathcal{V},
\end{eqnarray*}
where $\mathcal{V}=\{u\in \mathbb{R}^n : \|u\| \le 1\}$ is the uncertainty set and $g(x,v)=v^Tx -1$. It is easy to see that
the strict feasibility condition is satisfied.

We now show that our SDP relaxation is not exact. To see this, as $f$ is a convex homogeneous polynomial with degree at least $2$ (which is necessarily nonnegative), we first note that $\inf_{x\in \mathbb{R}^n}\{f(x): g(x,v) \le 0, \, \forall \, v\in \mathcal{V}\}=0$. The claim will follow if we show that, for any $(w_1^0,w_1^1,\ldots,w_1^n) \in \mathbb{R}^{n+1}$ with $\|(w_1^1,\ldots,w_1^n) \| \le w_1^0$, $f(x)+(-1)w_1^0+\sum_{i=1}^n w_1^i x_i-0   \notin \Sigma_d^2$. Otherwise, there exists a sum of squares polynomial $\sigma$ with degree at most $d$ such that
\begin{equation}\label{eq:representation}
f(x)+(-1)w_1^0+ \sum_{i=1}^n w_1^i x_i=\sigma(x), \  \mbox{ for all } x \in \mathbb{R}^n.
\end{equation}
for some $(w_1^0,w_1^1,\ldots,w_1^n) \in \mathbb{R}^{n+1}$ with $\|(w_1^1,\ldots,w_1^n) \| \le w_1^0$.
Note that $\sigma$ is a sum-of-squares (and so, is nonnegative) and $w_1^0 \ge 0$. So, $f(x) \ge h(x):=-\sum_{i=1}^n w_1^i x_i$.
 As $f$ is a convex homogeneous polynomial with degree $m$ and $m \ge 2$, $h \equiv 0$. (Indeed, if there exists $\hat{x}$
  such that $h(\hat{x}) \neq 0$, then by replacing $\hat{x}$ with $-\hat{x}$, we can assume that $h(\hat{x}) >0$. Now, we have
   $t^mf(\hat{x})=f(t \hat{x}) \ge 2h(t\hat{x})=2th (\hat{x})$,  for all $t > 0$. This shows us that $\frac{f(\hat{x})}{h(\hat{x})} \ge \frac{1}{t^{m-1}}$, for all $t>0$.
    This is a contradiction as $\frac{1}{t^{m-1}}\rightarrow \infty$ as $t \rightarrow 0$.) Hence, $f=\sigma+w_1^0$ and so $f$ is a sum-of-squares polynomial. This contradicts our construction of $f$. Therefore, our relaxation is not exact.
\end{example}

\begin{corollary}[{\bf Robust SOSCP with a sum of quadratic and separable functions}]
\label{cor431}
Consider the uncertain convex polynomial optimization problem under ellipsoidal data uncertainty $(P^e)$ and its relaxation problem $(D_k^e)$,
where each $g_i^{(r)}$ is the sum of a separable convex polynomial and a convex quadratic function, i.e., $g_i^{(r)}(x)=\sum_{l=1}^nh_{il}^{(r)}(x_l)+\frac{1}{2}x^TB_i^{(r)}x + \big(b_i^{(r)}\big)^T x + \beta_i^{(r)}$ for some convex univariate polynomial $h_{il}^{(r)}$, $B_i^{(r)} \succeq 0$, $b_i^{(r)} \in \mathbb{R}^n$ and $\beta_i^{(r)} \in \mathbb{R}$. Suppose that
 $\{x\in\mathbb{R}^n : g_i^{(0)}(x)+\sum_{r=1}^{q_i} v_i^{(r)} g_i^{(r)}(x)<0 \ \forall v_i\in \hat{\mathcal{V}}_i, i=1,\ldots,m\} \neq \emptyset.$
Then,  we have $\inf(P^e)=\max(D_{k_0}^e)$. 
\end{corollary}
\begin{proof}
The conclusion follows from the preceding theorem by noting that the sum of a separable convex polynomial and a convex quadratic function is a SOS-convex polynomial.
\end{proof}

\begin{corollary}[{\bf Robust convex quadratic problem}]
\label{cor:CQP}
For problem $(P^e)$ and its relaxation problem $(D_2^e)$, let
$f(x)= x^TAx + 2a^T x +\alpha$, $g_i^{(r)}(x) =x^TB_i^{(r)}x + 2\big(b_i^{(r)}\big)^T x + \beta_i^{(r)}$, $r=0,1,\ldots,t_i$, and  $g_i^{(r)}(x) =\big(b_i^{(r)}\big)^T x + \beta_i^{(r)}$, $r=t_i+1,\ldots,q_i$,
where $A,B_i^{(r)} \succeq 0$, $a,b_i^{(r)} \in \mathbb{R}^n$ and $\alpha,\beta_i^{(r)} \in \mathbb{R}$. Suppose that $\{x\in\mathbb{R}^n : g_i^{(0)}(x)+\sum_{r=1}^{q_i} v_i^{(r)} \ g_i^{(r)}(x)<0\ \forall v_i\in \hat{\mathcal{V}}_i, i=1,\ldots,m\} \neq \emptyset.$  Then,
$\inf(P^e)=\max(D_{2}^e)$ and $\max(D_2^e)$ can be written as the following semi-definite programming problem
\begin{eqnarray}\label{eq:ut}
& \max\limits_{\mu, w_i^r}   &  \mu \nonumber \\
&  s.t.   &  \begin{pmatrix}
A+\sum\limits_{i=1}^{m}\sum\limits_{r=0}^{t_i} w_i^{r}B_i^{(r)}       &   a+\sum\limits_{i=1}^{m}\sum\limits_{r=0}^{q_i} w_i^{r}b_i^{(r)} \\
(a+\sum\limits_{i=1}^{m}\sum\limits_{r=0}^{q_i} w_i^{r}b_i^{(r)})^T   &   \alpha+\sum\limits_{i=1}^{m}\sum\limits_{r=0}^{q_i} w_i^{r}\beta_i^{(r)}-\mu
\end{pmatrix} \succeq 0 \\
  &   &   \|(w_i^{1},\ldots,w_i^{t_i})\| \leq w_i^{0}, \ \forall i=1,\ldots,m, \nonumber \\
  &   &   \|(w_i^{t_i+1},\ldots,w_i^{q_i})\| \leq w_i^{0}, \ \forall i=1,\ldots,m, \nonumber \\
  &   &   w_i^r \geq 0, \ \forall r=0,1,\ldots,t_i, \ \forall i=1,\ldots,m, \nonumber \\
  &   &   \mu\in\mathbb{R}, w_i^r \in \mathbb{R}, \ \forall r=t_i+1\ldots,q_i,\ \forall i=1,\ldots,m. \nonumber
\end{eqnarray}
\end{corollary}

\begin{proof}
As $f$ and each $g_i^{(r)}$, $r=1,\ldots,q_i$, are convex quadratic functions,  the conclusion follows by applying Theorem \ref{cor43} and noting just that
$f + \sum\limits_{i=1}^{m}\sum\limits_{r=0}^{q_i} w_i^{r} g_i^{(r)} -\mu  \in \Sigma^2_{2}\,$ is, in this particular case, equivalent to
$ \begin{pmatrix}
A+\sum\limits_{i=1}^{m}\sum\limits_{r=0}^{t_i} w_i^{r}B_i^{(r)}       &   a+\sum\limits_{i=1}^{m}\sum\limits_{r=0}^{q_i} w_i^{r}b_i^{(r)} \\
(a+\sum\limits_{i=1}^{m}\sum\limits_{r=0}^{q_i} w_i^{r}b_i^{(r)})^T   &   \alpha+\sum\limits_{i=1}^{m}\sum\limits_{r=0}^{q_i} w_i^{r}\beta_i^{(r)}-\mu
\end{pmatrix} \succeq 0. $
\end{proof}

We now present a numerical example verifying the exact SDP relaxation for a robust SOS-convex polynomial optimization problem where the objective function is neither a quadratic function  nor a separable function.
\begin{example}{\bf (Exact SDP relaxation for a robust non-quadratic SOS-convex problem)} Consider the following robust SOSCP
\begin{eqnarray*}
(P_5)  &  \min  &  x_1^4+2x_1^2-2x_1x_2+x_2^2 \\
&\mbox{ s.t. }  & v_1x_1+v_2x_2 \le 1, \ \forall \ \|(v_1,v_2)\| \le 1.
\end{eqnarray*}
It is easy to verify that global solution of $(P_5)$ is $(0,0)$ with optimal value zero, and robust Slater condition is satisfied. The corresponding $4th$-order relaxation problem is given by 
\begin{equation*}
\max\limits_{\substack{\mu \in \mathbb{R}, \lambda \geq 0 \\ \|(v_1,v_2)\| \le 1}} \{ \mu : x_1^4+2x_1^2-2x_1x_2+x_2^2+\lambda (v_1x_1+v_2x_2-1)-\mu \in \Sigma_4^2\}.
\end{equation*}
It can be equivalently reformulated as the following semi-definite programming problem:  
\begin{eqnarray*}
 & \max\limits_{\mu, w, W} & \mu \\
 &  s.t.  &   W_{11}=-w^0-\mu, 2W_{12}=w^1, 2W_{13}=w^2, 2W_{23}+2W_{14}=-2,  \\
 &   &   2W_{16}+W_{33}=1, 2W_{15}+W_{22}=2, W_{55}=1, \\
 &   &   W_{ij}=0  \  \,\forall (i,j) \notin \{(2,2),(2,3),(3,2),(3,3),(5,5)\} \cup \{\cup_{j=1}^6 (1,j)\} \cup \{\cup_{j=1}^6 (j,1)\}, \\
 &   &  \|(w^1,w^2)\| \leq w^0,  \mu\in\mathbb{R}, w=(w^0,w^1,w^2)\in\mathbb{R}^3, W=(W_{ij})\in S^6_+.
\end{eqnarray*}
Let $\mu^*$ be the optimal value of the above SDP problem associated to a maximizer $(\mu^*,\hat{w},\hat{W})$. Since $\hat{W} \succeq 0$, then $\hat{W}_{11} \geq 0$ which implies $-\hat{w}^0 - \mu^* \geq 0$, and so, $\mu^* \leq -\hat{w}^0 \leq 0 $. On the other hand, define $\overline{W}\in S^6_+$ by $\overline{W}_{33} = \overline{W}_{55} = 1$, $\overline{W}_{22}=2$ and $\overline{W}_{23}=\overline{W}_{32}=-1$ and $\overline{W}_{ij}=0$ otherwise. Let $\overline{w}^0 = \overline{w}^1 = \overline{w}^2 = 0$ and $\overline{\mu}=0$. It is not hard to verify that $\overline{W} \succeq 0$ and $(\overline{\mu},\overline{w}, \overline{W})$ is a feasible point for the above SDP problem. So, $\mu^* \geq 0$. Thus, $\mu^*=0$ which shows that the SDP relaxation is exact.
\end{example}


We note that, in the special case of quadratically constrained optimization problem with linear objective function under restrictive ellipsoidal data uncertainty,  the linear matrix inequality
constraint in (\ref{eq:ut}) reduces to \[
\left(\begin{array}{cc}
\sum\limits_{i=1}^{m}\sum\limits_{r=0}^{t_i} w_i^{r}(L_i^{(r)})^TL_i^{(r)}       &   a+\sum\limits_{i=1}^{m}\sum\limits_{r=0}^{q_i} w_i^{r}b_i^{(r)} \\
(a+\sum\limits_{i=1}^{m}\sum\limits_{r=0}^{q_i} w_i^{r}b_i^{(r)})^T   &   \alpha+\sum\limits_{i=1}^{m}\sum\limits_{r=0}^{q_i}w_i^{r}\beta_i^{(r)}-\mu
\end{array}\right) \succeq 0,
\]
where $L_i^{(r)} \in \mathbb{R}^{n \times s}$, is a matrix  such that $B_i^{(r)}=L_i^{(r)}(L_i^{(r)})^T$. It then follows from \cite{Loboetal98} (see also \cite[page 277]{Nestrov})
that this linear matrix inequality can be equivalently written as second-order cone constraints. So, for a quadratically constrained optimization problem with linear objective function under restrictive ellipsoidal data uncertainty,
the sums-of-squares relaxation problem can be equivalently rewritten as a second order cone programming problem, and hence, exact second-order cone relaxation holds under the robust strict feasibility condition.

A second-order cone reformulation of a robust quadratic optimization problem with linear objective function under restrictive ellipsoidal data uncertainty was first shown in \cite{goldfab}. Our Corollary provides an alternative second-order cone reformulation for this class of problems.

\section{Conclusion}

In this paper, we studied robust solutions and semidefinite linear
programming (SDP) relaxations for SOS-convex polynomial optimization problems in the face of data uncertainty. We established sums-of-squares polynomial representations characterizing robust solutions and exact SDP-relaxations of robust SOS-convex optimization problem under various commonly used uncertainty sets.
It is easy to see from our SDP-relaxation results that the optimal value of a robust SOS-convex polynomial optimization problem can be found by solving a single semi-definite programming problem.

On the other hand, it is known that, using the Lasserre hierarchy  together with
a moment approach \cite{Lasserre}, we can get a sequence of points converging to a minimizer of the original polynomial optimization problem. Employing the moment approach of \cite{Lasserre,Lasserre_robust} and using our exact SDP relaxation results, one can get a sequence of points converging to a minimizer of a given robust SOS-convex polynomial optimization problem.

\end{document}